\documentclass{amsart}

\usepackage[latin1]{inputenc}
\usepackage[OT1]{fontenc}
\usepackage{color,fancyvrb}

\usepackage{bbm}
\usepackage{amsmath,amssymb,graphicx}
\usepackage{epsfig}
\usepackage{epic,eepic}

\newcommand{\DS}[1]{{\displaystyle #1}}

\newcommand{\nbN}{\mathbbm{N}}

\newcommand{\nbP}{\mathbbm{P}}

\newcommand{\nbE}{\mathbbm{E}}

\newcommand{\ten}{\longrightarrow}

\renewcommand{\Pr}{\nbP}
\renewcommand{\phi}{\varphi}

\newtheorem{theorem}{{\bf Theorem}}
\newtheorem{lemm}[theorem]{{\bf Lemma}}
\newtheorem{cor}[theorem]{{\bf Corollary}}

\newtheorem*{conj}{Conjecture}

\def\QED{\mbox{\rule[0pt]{1.5ex}{1.5ex}} \vspace{0.2cm}}
\def\cqfd{\hspace*{\fill}~\QED\par\endtrivlist\unskip}

\begin{document}

\begin{abstract}
We study in this paper a generalized coupon collector problem, which consists
in determining the distribution and the moments of the time needed to collect
a given number of distinct coupons that are drawn from a set of coupons
with an arbitrary probability distribution. We suppose that a special coupon
called the null coupon can be drawn but never belongs to any collection.
In this context, we obtain expressions of the distribution and the moments
of this time. We also prove that the almost-uniform distribution, for which all 
the non-null coupons have the same drawing probability, is the distribution
which minimizes the expected time to get a fixed subset of distinct coupons.
This optimization result is extended to the complementary distribution 
of that time when the full collection is considered, proving
by the way this well-known conjecture. Finally, we propose a new 
conjecture which expresses the fact that the almost-uniform distribution should 
minimize the complementary distribution of the time needed to get any fixed 
number of distinct coupons. 

\end{abstract}

\title{New results on a generalized coupon collector problem
       using Markov chains}
\author{Emmanuelle Anceaume, Yann Busnel and Bruno Sericola} 

\maketitle

\section{Introduction}
The coupon collector problem is an old problem which consists in evaluating
the time needed to get a collection 
of different objects drawn randomly
using a given probability distribution.  
This problem has given rise to a lot of attention
from researchers in various fields since it has applications in many 
scientific domains including computer science and optimization.

More formally, consider a set of $n$ coupons which are drawn randomly one by 
one, with replacement, coupon $i$ being drawn with probability $p_i$.
The classical coupon collector problem is         
to determine the expectation or the distribution of the number 
of coupons that need to be drawn from the set of $n$ coupons
to obtain the full collection of the $n$ coupons.
A large number of papers have been devoted to the analysis of asymptotics
and limit distributions of this distribution when $n$ tends to infinity,
see \cite{Doumas12} or \cite{Neal08} and the references therein. 
In \cite{Brown08}, the authors obtain some new formulas concerning this
distribution and they also provide simulation techniques to compute it as
well as analytic bounds of it.

We consider in this paper several generalizations of this problem.
A first generalization is the analysis, for $c \leq n$, of the number $T_{c,n}$ 
of coupons that need to be drawn, with replacement, to collect $c$ different
coupons
from set $\{1,2,\ldots,n\}$. With this notation, the number 
of coupons that need to be drawn from this set to obtain the full collection is
$T_{n,n}$.
If a coupon is drawn
at each discrete time $1,2,\ldots$ then $T_{c,n}$ is the time 
needed to obtain $c$ different coupons also called the waiting time to
obtain $c$ different coupons. This problem has been considered in 
\cite{Rubin65} in the case where the drawing probability distribution is 
uniform.

In a second generalization, we assume that $p=(p_1,\ldots,p_n)$ is not
necessarily a probability distribution, \emph{i.e.}, we suppose that
$p_1 + \cdots + p_n \leq 1$ and we define $p_0 = 1 - (p_1 + \cdots + p_n)$.
This means that there is a null coupon, denoted by $0$,
which is drawn with probability $p_0$, but which does not belong 
to the collection. 
In this context, the problem is to determine the distribution of the 
number $T_{c,n}$ of coupons that need to
be drawn from set $\{0,1,2,\ldots,n\}$, with replacement, till one first
obtains a collection composed of $c$ different coupons, $1 \leq c \leq n$,
among $\{1,\ldots,n\}$. 
These generalizations are motivated by the analysis of streaming algorithms 
in network monitoring applications as presented in 
Section~\ref{sec:application}.
 
The distribution of $T_{c,n}$ is obtained using Markov chains in Section 2,
in which we moreover
show that this distribution leads to new combinatorial identities.
This result is used to get an expression of $T_{c,n}(v)$ when the drawing 
distribution is the almost-uniform distribution denoted by $v$ and defined
by $v=(v_1,\ldots,v_n)$ with $v_i = (1-v_0)/n$, where 
$v_0 = 1 -(v_1 + \cdots +v_n)$.
Expressions of the moments of $T_{c,n}(p)$ are given in 
Section 3, where we show that the limit of $\nbE(T_{c,n}(p))$ is equal to $c$
when $n$ tends to infinity.
We show in Section 4 that the almost-uniform distribution $v$ and 
the uniform distribution $u$
minimize the expected value $\nbE(T_{c,n}(p))$.
We prove in Section 5 that 
the tail distribution of $T_{n,n}$ is minimized over all the $p_1,\ldots,p_n$
by the almost-uniform distribution and by the uniform distribution. 
This result was expressed as a conjecture
in the case where $p_0=0$, \emph{i.e.}, when $p_1 + \cdots + p_n = 1$, 
in several papers like \cite{Boneh97} for instance, from which the idea of 
the proof comes from.
We propose in Section 6 a new conjecture which consists in showing that
the distributions $v$ and $u$ minimize the tail distribution of 
$T_{c,n}(p)$. 
This conjecture is motivated by the fact that it is true 
for $c=1$ and $c=n$ as shown in Section 5, and we show that 
it is also true for $c=2$. It is moreover true for
the expected value $\nbE(T_{c,n}(p))$ as shown in Section 4.

\section{Distribution of $T_{c,n}$}
Recall that $T_{c,n}$ is the number of coupons that need to
be drawn from set $\{0,1,2,\ldots,n\}$, with replacement, till one first
obtains a collection with $c$ different coupons, $1 \leq c \leq n$,
among $\{1,\ldots,n\}$, where coupon $i$ is drawn with probability $p_i$,
$i=0,1,\ldots,n$.

To obtain the distribution of $T_{c,n}$, we consider the discrete-time 
Markov chain 
$X = \{X_m, \; m \geq 0\}$ that represents the collection obtained after
having drawn $m$ coupons. The state space of $X$ is  
$$S_n = \left\{J \subseteq \{1,\ldots,n\}\right\}$$
with transition probability matrix, denoted by $Q$, given,
for every $J,H \in S_n$, by
$$Q_{J,H} = \left\{\begin{array}{cl}
            p_\ell & \mbox{if } H \setminus J = \{\ell\} \\
            p_0 + P_J & \mbox{if } J=H \\
            0 & \mbox{otherwise}, \\  
\end{array}\right.$$
where, for every $J \in S_n$, $P_J$ is given by 
\begin{equation} \label{PJ}
P_J = \sum_{j \in J} p_j, 
\end{equation}
with $P_\emptyset = 0$.
It is easily checked that Markov chain $X$ is acyclic, \emph{i.e.}, 
it has no cycle
of length greater than $1$, and that all the states are transient, except
state $\{1,\ldots,n\}$ which is absorbing.
We introduce the partition $(S_{0,n},S_{1,n}, \ldots, S_{n,n})$ of $S_n$, 
where $S_{i,n}$ is defined, for $i=0,\ldots,n$, by 
\begin{equation} \label{Si}
S_{i,n} = \left\{J \subseteq \{1,\ldots,n\} \mid |J| = i \right\}.
\end{equation}
Note that we have $S_{0,n} = \{\emptyset\}$,
$$|S_n| = 2^n \mbox{ and } |S_{i,n}| = {n \choose i}.$$
Assuming that $X_0 = \emptyset$ with probability $1$, the random variable 
$T_{c,n}$ can then be defined, for every $c=1, \ldots,n$, by
$$T_{c,n} = \inf\{m \geq 0 \mid X_m \in S_{c,n}\}.$$
The distribution of $T_{c,n}$ is obtained in Theorem \ref{theo1}
using the Markov property and the following lemma.

For every $n \geq 1$, $\ell = 1,\ldots,n$ and $i=0,\ldots,n$, 
we define the set $S_{i,n}(\ell)$ by
$$S_{i,n}(\ell) = \left\{J \subseteq \{1,\ldots,n\}\setminus\{\ell\} 
\mid |J| = i\right\}.$$
\begin{lemm} \label{lemme1}
For every $n \geq 1$, for every $k \geq 0$, 
for all positive real numbers $y_1,\ldots,y_n$,
for every $i = 1,\ldots,n$ and all real number $a \geq 0$, we have
$$\sum_{\ell = 1}^n y_\ell
\sum_{J \in S_{i-1,n}(\ell)} (a + y_\ell + Y_J)^{k}
= \sum_{J \in S_{i,n}} Y_J(a + Y_J)^{k},$$
where $Y_J= \sum_{j \in J} y_j$ and $Y_\emptyset = 0$. 
\end{lemm}

\begin{proof}
For $n=1$, since $S_{0,1}(1) = \emptyset$, the left hand side is
equal to $y_1 (a + y_1)^k$ and since $S_{1,1} = \{1\}$, 
the right hand side is also equal to $y_1 (a + y_1)^k$.
Suppose that the result is true for integer $n-1$, with $n \geq 2$, 
\emph{i.e.}, suppose that
for every $k \geq 0$, 
for all positive real numbers $y_1,\ldots,y_{n-1}$,
for every $i = 1,\ldots,n-1$ and for all real number $a \geq 0$, we have
$$\sum_{\ell = 1}^{n-1} y_\ell
\sum_{J \in S_{i-1,n-1}(\ell)} (a + y_\ell + Y_J)^{k}
= \sum_{J \in S_{i,n-1}} Y_J(a + Y_J)^{k}.$$
We then have
$$\sum_{\ell = 1}^{n} y_\ell
\sum_{J \in S_{i-1,n}(\ell)} (a + y_\ell + Y_J)^{k}
= \sum_{\ell = 1}^{n-1} y_\ell
\sum_{J \in S_{i-1,n}(\ell)} (a + y_\ell + Y_J)^{k} +
y_n \sum_{J \in S_{i-1,n}(n)} (a + y_n + Y_J)^{k}.$$
Since $S_{i-1,n}(n) = S_{i-1,n-1}$, we get
$$\sum_{\ell = 1}^{n} y_\ell
\sum_{J \in S_{i-1,n}(\ell)} (a + y_\ell + Y_J)^{k}
= \sum_{\ell = 1}^{n-1} y_\ell
\sum_{J \in S_{i-1,n}(\ell)} (a + y_\ell + Y_J)^{k} +
y_n \sum_{J \in S_{i-1,n-1}} (a + y_n + Y_J)^{k}.$$
For $\ell = 1,\ldots,n-1$, the set $S_{i-1,n}(\ell)$ can be partitioned 
into two subsets $S'_{i-1,n}(\ell)$ and $S''_{i-1,n}(\ell)$
defined by
$$S'_{i-1,n}(\ell) = \left\{J \subseteq \{1,\ldots,n\}\setminus\{\ell\} 
\mid |J| = i-1 \mbox{ and } n \in J \right\}$$
and
$$S''_{i-1,n}(\ell) = \left\{J \subseteq \{1,\ldots,n\}\setminus\{\ell\} 
\mid |J| = i-1 \mbox{ and } n \notin J \right\}.$$
Since $S''_{i-1,n}(\ell) = S_{i-1,n-1}(\ell)$, the previous relation becomes
\begin{align*}
& \sum_{\ell = 1}^{n} y_\ell
\sum_{J \in S_{i-1,n}(\ell)} (a + y_\ell + Y_J)^{k} \\
& = \sum_{\ell = 1}^{n-1} y_\ell
\left[\sum_{J \in S_{i-1,n-1}(\ell)} (a + y_\ell + Y_J)^{k} +
\sum_{J \in S'_{i-1,n}(\ell)} (a + y_\ell + Y_J)^{k}\right] +
y_n \sum_{J \in S_{i-1,n-1}} (a + y_n + Y_J)^{k} \\
& = \sum_{\ell = 1}^{n-1} y_\ell
\sum_{J \in S_{i-1,n-1}(\ell)} (a + y_\ell + Y_J)^{k} +
\sum_{\ell = 1}^{n-1} y_\ell
\sum_{J \in S_{i-2,n-1}(\ell)} (a+y_n + y_\ell + Y_J)^{k} +
y_n \sum_{J \in S_{i-1,n-1}} (a + y_n + Y_J)^{k}.
\end{align*}
The recurrence hypothesis can be applied for both the first and
the second terms. For the second term, the constant $a$ is replaced by the 
constant $a+y_n$.
We thus obtain
\begin{align*}
\sum_{\ell = 1}^{n} y_\ell
\sum_{J \in S_{i-1,n}(\ell)} & (a + y_\ell + Y_J)^{k} \\
& = \sum_{J \in S_{i,n-1}} Y_J(a + Y_J)^{k} +
\sum_{J \in S_{i-1,n-1}} Y_J(a + y_n + Y_J)^{k} +
y_n \sum_{J \in S_{i-1,n-1}} (a + y_n + Y_J)^{k} \\
& = \sum_{J \in S_{i,n-1}} Y_J(a + Y_J)^{k} +
\sum_{J \in S_{i-1,n-1}} (y_n + Y_J)(a + y_n + Y_J)^{k} \\
& = \sum_{J \in S_{i,n-1}} Y_J(a + Y_J)^{k} +
\sum_{J \in S'_{i,n}} Y_J(a + Y_J)^{k},
\end{align*}
where $S'_{i,n} = \left\{J \subseteq \{1,\ldots,n\} 
\mid |J| = i \mbox{ and } n \in J \right\}$.

Consider the set $S''_{i,n} = \left\{J \subseteq \{1,\ldots,n\} 
\mid |J| = i \mbox{ and } n \notin J \right\}$.
The sets $S'_{i,n}$ and $S''_{i,n}$ form a partition of $S_{i,n}$ and since
$S''_{i,n} = S_{i,n-1}$, we get
\begin{eqnarray*}
\sum_{\ell = 1}^{n} y_\ell
\sum_{J \in S_{i-1,n}(\ell)} (a + y_\ell + Y_J)^{k}
& = & \sum_{J \in S_{i,n-1}} Y_J(a + Y_J)^{k} +
\sum_{J \in S'_{i,n}} Y_J(a + Y_J)^{k} \\
& = & \sum_{J \in S''_{i,n}} Y_J(a + Y_J)^{k} +
\sum_{J \in S'_{i,n}} Y_J(a + Y_J)^{k} \\
& = & \sum_{J \in S_{i,n}} Y_J(a + Y_J)^{k},
\end{eqnarray*}
which completes the proof.
\end{proof}
In the following we will use the fact that the distribution of $T_{c,n}$
depends on the vector $p=(p_1,\ldots,p_n)$, so we will use the notation
$T_{c,n}(p)$ instead of $T_{c,n}$, meaning by the way that vector $p$ is 
of dimension $n$. We will also use the notation 
$$p_0 = 1 - \sum_{i=1}^n p_i.$$ 
Finally, for $\ell = 1,\ldots,n$, the notation $p^{(\ell)}$ will denote
the vector $p$ in which the entry $p_\ell$ has been removed, that is
$p^{(\ell)} = (p_i)_{1 \leq i \leq n, i \neq \ell}$. The dimension
of $p^{(\ell)}$, which is $n-1$ here, is not specified but will be clear 
by the context of its use.
We are now able to prove the following result.

\begin{theorem} \label{theo1}
For every $n \geq 1$ and $c=1,\ldots,n$,
we have, for every $k \geq 0$,
\begin{equation} \label{loiTc}
\Pr\{T_{c,n}(p) > k\} = \sum_{i=0}^{c-1} (-1)^{c-1-i} {n-i-1 \choose n-c}
\sum_{J \in S_{i,n}} (p_0 + P_J)^k,
\end{equation}
where $P_J$ is given by (\ref{PJ}).
\end{theorem}

\begin{proof}
Relation (\ref{loiTc}) is true for $c=1$ since in this case we have
$$\Pr\{T_{1,n}(p) > k\} = p_0^k.$$
So we suppose now that $n \geq 2$ and $c=2,\ldots,n$.

Since $X_0 = \emptyset$, conditioning on $X_1$ and using the Markov property,
see for instance \cite{Sericola13}, we get for $k \geq 1$, 
\begin{equation} \label{ccdf}
\Pr\{T_{c,n}(p) > k\} = 
p_0 \Pr\{T_{c,n}(p) > k-1\} + \sum_{\ell = 1}^n
p_\ell \Pr\{T_{c-1,n-1}(p^{(\ell)}) > k-1\}.
\end{equation}
We now proceed by recurrence over index $k$.
Relation (\ref{loiTc}) is true for $k = 0$ since it is well-known that
\begin{equation} \label{combi}
\sum_{i=0}^{c-1} (-1)^{c-1-i} {n-i-1 \choose n-c}{n \choose i} = 1.
\end{equation}
Relation (\ref{loiTc}) is also true for $k=1$ since on the one hand 
$\Pr\{T_{c,n}(p) > 1\} = 1$
and on the other hand, using Relation (\ref{ccdf}), we have
\begin{eqnarray*}
\Pr\{T_{c,n}(p) > 1\}  
& = & p_0 \Pr\{T_{c,n}(p) > 0\} + \sum_{\ell = 1}^n
p_\ell \Pr\{T_{c-1,n-1}(p^{(\ell)}) > 0\} \\
& = & p_0 + \sum_{\ell = 1}^n p_\ell \\
& = & 1.
\end{eqnarray*}
Suppose now that Relation (\ref{loiTc}) is true for integer $k-1$, that is, 
suppose that we have
$$\Pr\{T_{c,n}(p) > k-1\} = 
\sum_{i=0}^{c-1} (-1)^{c-1-i} {n-i-1 \choose n-c}
\sum_{J \in S_{i,n}} (p_0 + P_J)^{k-1}.$$
Using (\ref{ccdf}) and the recurrence relation, we have
\begin{align*}
\Pr\{T_{c,n}(p) > k\} & = 
p_0 \sum_{i=0}^{c-1} (-1)^{c-1-i} {n-i-1 \choose n-c}
\sum_{J \in S_{i,n}} (p_0 + P_J)^{k-1} \\
& \;\;\; + 
\sum_{\ell = 1}^n p_\ell \sum_{i=0}^{c-2} (-1)^{c-2-i} {n-i-2 \choose n-c}
\sum_{J \in S_{i,n}(\ell)} (p_0 + p_\ell + P_J)^{k-1}. 
\end{align*}
Using the change of variable $i:=i-1$ in the second sum, we obtain 
\begin{align*}
\Pr\{T_{c,n}(p) > k\} & = 
\sum_{i=0}^{c-1} (-1)^{c-1-i} {n-i-1 \choose n-c}
p_0 \sum_{J \in S_{i,n}} (p_0 + P_J)^{k-1} \\
& \;\;\; + 
\sum_{i=1}^{c-1} (-1)^{c-1-i} {n-i-1 \choose n-c}
\sum_{\ell = 1}^n p_\ell \sum_{J \in S_{i-1,n}(\ell)} 
(p_0 + p_\ell + P_J)^{k-1} \\
& = \sum_{i=1}^{c-1} (-1)^{c-1-i} {n-i-1 \choose n-c}
\bigg[p_0 \sum_{J \in S_{i,n}} (p_0 + P_J)^{k-1} + \sum_{\ell = 1}^n p_\ell
\sum_{J \in S_{i-1,n}(\ell)} (p_0 + p_\ell + P_J)^{k-1}\bigg] \\
& \;\;\; + (-1)^{c-1} {n-1 \choose n-c} p_0^{k}.
\end{align*}
From Lemma \ref{lemme1}, we have
$$\sum_{\ell = 1}^n p_\ell
\sum_{J \in S_{i-1,n}(\ell)} (p_0 + p_\ell + P_J)^{k-1}
= \sum_{J \in S_{i,n}} P_J (p_0 + P_J)^{k-1},$$
that is 
\begin{eqnarray*}
\Pr\{T_{c,n}(p) > k\} 
& = & (-1)^{c-1} {n-1 \choose n-c} p_0^{k} +
\sum_{i=1}^{c-1} (-1)^{c-1-i} {n-i-1 \choose n-c}
\sum_{J \in S_{i,n}} (p_0 + P_J)^k \\
& = & \sum_{i=0}^{c-1} (-1)^{c-1-i} {n-i-1 \choose n-c}
\sum_{J \in S_{i,n}} (p_0 + P_J)^k,
\end{eqnarray*}
which completes the proof.
\end{proof}
This theorem also shows, as expected, that the function 
$\Pr\{T_{c,n}(p) > k\}$, as a function
of $p$, is symmetric, which means that it has the same value for
any permutation of the entries of $p$.
As a corollary, we obtain the following combinatorial identities.

\begin{cor} \label{cor1}
For every $c \geq 1$, for every $n \geq c$ and for all 
$p_1,\ldots,p_n \in (0,1)$ such that $p_1+\cdots+p_n = 1$, we have
$$\sum_{i=0}^{c-1} (-1)^{c-1-i} {n-i-1 \choose n-c}
\sum_{J \in S_{i,n}} (p_0 + P_J)^{k-1} = 1, \mbox{ for } k=0,1, \ldots,c-1.$$
\end{cor}

\begin{proof}
The random variable $T_{c,n}$ takes its values on the set $\{c,c+1,\ldots\}$, 
so we have
$$\Pr\{T_{c,n} > k\} = 1, \mbox{ for } k=0,1, \ldots, c-1,$$
which completes the proof thanks to Theorem \ref{theo1}.
\end{proof}

For every $n \geq 1$ and for every $v_0 \in [0,1]$, we define the
vector $v = (v_1,\ldots,v_n)$ by $v_i = (1-v_0)/n$. We will refer it to as
the almost-uniform distribution.
We then have, from (\ref{loiTc}),
$$\Pr\{T_{c,n}(v) > k\} = \sum_{i=0}^{c-1} (-1)^{c-1-i} {n-i-1 \choose n-c}
{n \choose i} \left(v_0\left(1 - \frac{i}{n}\right) + \frac{i}{n}\right)^k.$$
We denote by $u=(u_1,\ldots,u_n)$ the uniform distribution defined by 
$u_i = 1/n$. It is equal to $v$ when $v_0=0$. The dimensions of $u$ and
$v$ are specified by the context.

\section{Moments of $T_{c,n}$}
For $r \geq 1$, the $r$th moment of $T_{c,n}(p)$ is defined by
$$\nbE(T^r_{c,n}(p)) = \sum_{k = 1}^\infty k^r \Pr\{T_{c,n}(p) = k\}.$$
It can be obtained in function of the tail distribution of $T_{c,n}(p)$ 
by writing
\begin{eqnarray*}
\nbE(T^r_{c,n}(p)) 
& = & \sum_{k = 1}^\infty k^r \Pr\{T_{c,n}(p) = k\} \\
& = & \sum_{k = 1}^\infty k^r \Pr\{T_{c,n}(p) > k-1\} 
      - \sum_{k = 1}^\infty k^r \Pr\{T_{c,n}(p) > k\} \\ 
& = & \sum_{k = 0}^\infty \left((k+1)^r - k^r\right) \Pr\{T_{c,n}(p) > k\} \\
& = & \sum_{\ell=0}^{r-1} {r \choose \ell} 
       \sum_{k = 0}^\infty k^\ell \Pr\{T_{c,n}(p) > k\}.
\end{eqnarray*}
We easily get the first and second moments of $T_{c,n}(p)$, 
by taking $r=1$ and $r=2$ respectively, that is 
\begin{equation} \label{ETc}
\nbE(T_{c,n}(p)) = \sum_{k=0}^\infty \Pr\{T_{c,n}(p) > k\}
= \sum_{i=0}^{c-1} (-1)^{c-1-i} {n-i-1 \choose n-c}
\sum_{J \in S_{i,n}} \frac{1}{1 - (p_0+P_J)}
\end{equation}
and
$$\nbE(T^2_{c,n}(p)) = \nbE(T_{c,n}(p)) + 
2\sum_{k=1}^\infty k\Pr\{T_{c,n}(p) > k\}
= \sum_{i=0}^{c-1} (-1)^{c-1-i} {n-i-1 \choose n-c}
\sum_{J \in S_{i,n}} \frac{1+2(p_0+P_J)}{[1 - (p_0+P_J)]^2}.$$
The expected value given by (\ref{ETc}) has been obtained in \cite{Flajolet92}
in the particular case where $p_0=0$.

When the drawing probabilities are given by the almost-uniform distribution $v$,
we get
\begin{eqnarray*}
\nbE(T_{c,n}(v)) 
& = & \frac{1}{1-v_0}
\sum_{i=0}^{c-1} (-1)^{c-1-i} {n-i-1 \choose n-c} {n \choose i} 
\frac{n}{n-i} \\
& = & \frac{1}{1-v_0} \nbE(T_{c,n}(u)).
\end{eqnarray*}
Using the following two relations 
$${n \choose i} = {n-1 \choose i} + {n-1 \choose i-1}1_{\{i \geq 1\}}
\mbox{ and } {n-1 \choose i} \frac{n}{n-i} = {n \choose i},$$
where $1_A$ is the indicator function of set $A$,
we get
\begin{eqnarray*}
\nbE(T_{c,n}(u))
& = & \sum_{i=0}^{c-1} (-1)^{c-1-i} {n-i-1 \choose n-c} {n \choose i} 
\frac{n}{n-i} \\
& = & \sum_{i=0}^{c-1} (-1)^{c-1-i} {n-i-1 \choose n-c} {n \choose i} +
\sum_{i=1}^{c-1} (-1)^{c-1-i} {n-i-1 \choose n-c} {n-1 \choose i-1}
\frac{n}{n-i}.
\end{eqnarray*}
From Relation (\ref{combi}), the first sum is equal to $1$. Using the change
of variable $i:=i+1$ in the second sum, we obtain
\begin{eqnarray}
\nbE(T_{c,n}(u)) 
& = & 1 + \sum_{i=0}^{c-2} (-1)^{c-2-i} {n-i-2 \choose n-c} {n-1 \choose i}
\frac{n}{n-i+1} \nonumber \\
& = & 1 + \frac{n}{n-1}\nbE(T_{c-1,n-1}(u)). \label{ETcnu}
\end{eqnarray}
Note that the dimension of the uniform distribution in the left hand side is 
equal to $n$
and the one in the right hand side is equal to $n-1$.
Since $\nbE(T_{1,n}(u)) = 1$, we obtain
\begin{equation} \label{ETcnvu}
\nbE(T_{c,n}(u)) = n(H_n - H_{n-c})
\mbox{ and } \nbE(T_{c,n}(v)) = \frac{n(H_n - H_{n-c})}{1 - v_0},
\end{equation}
where $H_\ell$ is the $\ell$th harmonic number defined by 
$H_0=0$ and,
for $\ell \geq 1$, 
$$H_\ell = \sum_{i=1}^\ell 1/i.$$
We deduce easily from (\ref{ETcnu}) that, for every $c \geq 1$, we have
$$\lim_{n \ten \infty} \nbE(T_{c,n}(u)) = c
\mbox{ and } \lim_{n \ten \infty} \nbE(T_{c,n}(v)) = \frac{c}{1 - v_0}.$$

In the next section we show that, when $p_0$ is fixed, the minimum value
of $\nbE(T_{c,n}(p))$ is reached when $p=v$, with $v_0 = p_0$.

\section{Distribution minimizing $\nbE(T_{c,n}(p))$}
The following lemma will be used to prove the next theorem.

\begin{lemm} \label{lemme2}
For every $n \geq 1$ and $r_1,\ldots,r_n \in (0,1)$ with
$r_1 + \cdots + r_n = 1$, we have
$$\sum_{\ell=1}^n \frac{1}{r_\ell} \geq n^2.$$
\end{lemm}

\proof
We proceed by recurrence. The result is clearly true for $n=1$.
Suppose that the result is true for integer $n-1$, $n \geq 2$.
We then have
$$\sum_{\ell=1}^n \frac{1}{r_\ell} 
= \frac{1}{r_n} + \sum_{\ell=1}^{n-1} \frac{1}{r_\ell} 
= \frac{1}{r_n} + \frac{1}{1-r_n} \sum_{\ell=1}^{n-1} \frac{1}{h_\ell},$$
where $h_\ell$ is given, for $\ell=1,\ldots n-1$, by
$$h_\ell = \frac{r_\ell}{1-r_n}.$$
Since $h_1 + \cdots + h_{n-1} = 1$, we get, using the recurrence hypothesis,
$$\sum_{\ell=1}^n \frac{1}{r_\ell}
\geq \frac{1}{r_n} + \frac{(n-1)^2}{1-r_n}
= \frac{(nr_n - 1)^2}{r_n(1-r_n)} + n^2 \geq n^2,$$
which completes the proof.
\cqfd

\begin{theorem} \label{theocn}
For every $n \geq 1$, for every $c=1,\ldots,n$ and
$p=(p_1,\ldots,p_n) \in (0,1)^n$ with
$p_1 + \cdots + p_n \leq 1$, we have
$$\nbE(T_{c,n}(p)) \geq \nbE(T_{c,n}(v)) \geq \nbE(T_{c,n}(u)),$$
where
$v = (v_1,\ldots,v_n)$ with $v_i = (1-p_0)/n$ and
$p_0 = 1 - (p_1 + \cdots + p_n)$ and where
$u = (1/n,\ldots,1/n)$.
\end{theorem}

\begin{proof}
The second inequality comes from (\ref{ETcnvu}).

Defining $v_0 = 1 - (v_1 + \cdots + v_n)$, we have $v_0 = p_0$.
For $c=1$, the result is trivial since we have from
Relation (\ref{ETc})
$$\nbE(T_{1,n}(p)) = \frac{1}{1 - p_0} = \frac{1}{1 - v_0} = \nbE(T_{1,n}(v)).$$

For $c \geq 2$, which implies that $n \geq 2$,
summing Relation (\ref{ccdf}) for $k \geq 1$, we get
$$\nbE(T_{c,n}(p)) - 1 = p_0 \nbE(T_{c,n}(p)) + \sum_{\ell = 1}^n
p_\ell \nbE(T_{c-1,n-1}(p^{(\ell)})).$$
We then obtain
\begin{equation} \label{ETcn}
\nbE(T_{c,n}(p)) = \frac{1}{1 - p_0}\left(1 + \sum_{\ell = 1}^n
p_\ell \nbE(T_{c-1,n-1}(p^{(\ell)}))\right).
\end{equation}
We now proceed by recurrence.
Suppose that the inequality is true for integer $c-1$, with $c \geq 2$, 
\emph{i.e.},
suppose that, for every $n \geq c$, for every
$q = (q_1,\ldots,q_{n-1}) \in (0,1)^{n-1}$ with
$q_1 + \cdots + q_{n-1} \leq 1$, we have
$$\nbE(T_{c-1,n-1}(q)) \geq \nbE(T_{c-1,n-1}(v)), \mbox{ with }
v_0 = q_0 = 1 - \sum_{i=1}^{n-1} q_i.$$
Using Relation (\ref{ETcnvu}), this implies that
$$\nbE(T_{c-1,n-1}(p^{(\ell)})) \geq 
\frac{(n-1)(H_{n-1} - H_{n-c})}{1 - (p_0 + p_\ell)}.$$
From Relation (\ref{ETcn}), we obtain
\begin{equation} \label{nbE}
\nbE(T_{c,n}(p)) \geq \frac{1}{1 - p_0}
\left(1 + (n-1)(H_{n-1} - H_{n-c}) \sum_{\ell = 1}^n 
\frac{p_\ell}{1 -(p_0 + p_\ell)}\right).
\end{equation}
Observe now that for $\ell=1,\ldots,n$ we have
$$\frac{p_\ell}{1-(p_0+p_\ell)} = -1 + \frac{1}{(n-1)r_\ell},$$
where the $r_\ell$ are given by
$$r_\ell = \frac{1-(p_0+p_\ell)}{(n-1)(1-p_0)}$$
and satisfy $r_1,\ldots,r_n \in (0,1)$ with
$r_1 + \cdots + r_n = 1$. From Lemma \ref{lemme2}, we obtain
$$\sum_{\ell = 1}^n \frac{p_\ell}{1-(p_0+p_\ell)} 
= -n + \frac{1}{n-1} \sum_{\ell = 1}^n \frac{1}{r_\ell}
\geq -n + \frac{n^2}{n-1} = \frac{n}{n-1}.$$
Replacing this value in (\ref{nbE}), we obtain, using (\ref{ETcnvu}),
$$\nbE(T_{c,n}(p)) \geq 
\frac{1}{1 - p_0}
\left(1 + n(H_{n-1} - H_{n-c})\right) = 
\frac{n(H_{n} - H_{n-c})}{1 - p_0} = \nbE(T_{c,n}(v)),$$
which completes the proof.
\end{proof}

\section{Distribution minimizing the distribution of $T_{n,n}(p)$}
For every $n \geq 1$, $i=0,1,\ldots,n$ and $k \geq 0$,
we denote by $N_i^{(k)}$ the number of coupons of type $i$
collected at instants $1,\ldots,k$. It is well-known that the joint 
distribution of the $N_i^{(k)}$ is a multinomial distribution. More precisely,
for every $k \geq 0$ and $k_0, k_1,\ldots,k_n \geq 0$ such that
$k_0 + k_1 + \cdots + k_n = k$, we have
\begin{equation} \label{multinome}
\Pr\{N_0^{(k)} = k_0, N_1^{(k)} = k_1,\ldots,N_n^{(k)} = k_n\} = 
\frac{k!}{k_0!k_1!\cdots k_n!} p_0^{k_0}p_1^{k_1} \cdots p_n^{k_n}.
\end{equation}
Recall that the coupons of type $0$ do not belong to the collection.
For every $\ell = 1,\ldots,n$, we easily deduce that, for every $k \geq 0$
and $k_1,\ldots,k_\ell \geq 0$ such that
$k_1 + \cdots + k_\ell \leq k$,
$$\Pr\{N_1^{(k)} = k_1,\ldots,N_\ell^{(k)} = k_\ell\} = 
\frac{k!}{k_1!\cdots k_\ell!\left(k - (k_1 + \cdots + k_\ell)\right)!} 
p_1^{k_1} \cdots p_\ell^{k_\ell}
\left(1 - (p_1 + \cdots + p_\ell)\right)^{k-(k_1 + \cdots + k_\ell)}.$$
To prove the next theorem, we recall some basic results on convex functions.
A function $f$ is said to be convex on an interval $I$
if for every $x,y \in I$ and $\lambda \in [0,1]$, we have
$$f(\lambda x + (1-\lambda)y) \leq \lambda f(x)+(1-\lambda)f(y).$$
Let $f$ be a function defined on an interval $I$.
For every $\alpha \in I$, we introduce the function $g_\alpha$, 
defined for every $x \in I\setminus\{\alpha\}$, by
$$g_\alpha(x) = \frac{f(x)-f(\alpha)}{x-\alpha}.$$
It is well-known that $f$ is a convex function on interval $I$
if and only if for every $\alpha \in I$, the function $g_\alpha$
is increasing on $I\setminus\{\alpha\}$.
The next result is also known but less popular, so we give its proof.

\begin{lemm} \label{lemmconvex}
Let $f$ be a convex function on an interval $I$.
For every $x,y,z,t \in I$ with $x < y,z <t$, we have
$$(t-y)f(z) + (z-x)f(y) \leq (t-y)f(x) + (z-x)f(t).$$
If, moreover, we have $t+x=y+z$, we get
$$f(z) + f(y) \leq f(x) + f(t).$$
\end{lemm}

\begin{proof}
It suffices to apply twice the property that function $g_\alpha$ is increasing
on $I\setminus\{\alpha\}$, for every $\alpha \in I$.
Since $z < t$, we have $g_x(z) \leq g_x(t)$ and since $x <y$, we have
$g_t(x) \leq g_t(y)$. But as $g_x(t) = g_t(x)$ and $g_t(y) = g_y(t)$, we obtain
$g_x(z) \leq g_x(t) = g_t(x) \leq g_t(y)=g_y(t)$, which means in particular 
that
$$\frac{f(z)-f(x)}{z-x} \leq \frac{f(t)-f(y)}{t-y},$$
that is 
$$(t-y)f(z) + (z-x)f(y) \leq (t-y)f(x) + (z-x)f(t).$$
The rest of the proof is trivial since $t+x=y+z$ implies that $t-y=z-x>0$.
\end{proof}

\begin{theorem} \label{autreloi}
For every $n \geq 1$ and $p=(p_1,\ldots,p_n) \in (0,1)^n$ with
$p_1 + \cdots + p_n \leq 1$, we have, for every $k \geq 0$,
$$\Pr\{T_{n,n}(p') \leq k\} \leq \Pr\{T_{n,n}(p) \leq k\},$$
where $p'=(p_1,\ldots,p_{n-2},p'_{n-1},p'_n)$
with
$p'_{n-1} = \lambda p_{n-1} + (1-\lambda) p_n$ and
$p'_{n} = (1-\lambda) p_{n-1} + \lambda p_n$,
for every $\lambda \in [0,1]$. 
\end{theorem}

\begin{proof}
If $\lambda = 1$ then we have $p'=p$ so the result is trivial.
If $\lambda = 0$ then we have $p'_{n-1} =  p_n$ and $p'_{n} = p_{n-1}$ and
the result is also trivial since the function $\Pr\{T_{n,n}(p) \leq k\}$ is
a symmetric function of $p$.
We thus suppose now that $\lambda \in (0,1)$.

For every $n \geq 1$ and $k \geq 0$, we have
$$\{T_{n,n}(p) \leq k\} = \{N_1^{(k)} > 0,\ldots,N_n^{(k)} > 0\}.$$
We thus get, for $k_1 > 0,\ldots,k_{n-2}>0$ such that 
$k_1 + \cdots + k_{n-2} \leq k$,
setting $s=k-(k_1 + \cdots + k_{n-2})$,
\begin{align*}
\Pr\{T_{n,n}(p) \leq k, & \; N_1^{(k)}=k_1,\ldots,N_{n-2}^{(k)}=k_{n-2}\} \\
& = \; \Pr\{N_1^{(k)}=k_1,\ldots,N_{n-2}^{(k)}=k_{n-2},
N_{n-1}^{(k)} > 0,N_{n}^{(k)} >0\} \\
& = \; \sum_{
\begin{array}{c}
\scriptstyle{u \geq 0, v>0, w > 0,} \\ [-5pt]
\scriptstyle{u + v +w = s} \\
\end{array}
}
\Pr\{N_0^{(k)} = u,N_1^{(k)}=k_1,\ldots,N_{n-2}^{(k)}=k_{n-2},
N_{n-1}^{(k)} = v, N_n^{(k)} = w\}.
\end{align*}
Using Relation (\ref{multinome}) and introducing the notation
$$q_0 = \frac{p_0}{p_0+p_{n-1}+p_n}, \;
q_{n-1} = \frac{p_{n-1}}{p_0+p_{n-1}+ p_n} \mbox{ and }
q_n = \frac{p_{n}}{p_0+p_{n-1}+p_n},$$
we obtain
\begin{align*}
\Pr\{T_{n,n}(p) \leq k, & \; N_1^{(k)}=k_1,\ldots,N_{n-2}^{(k)}=k_{n-2}\} \\
& = \sum_{
\begin{array}{c}
\scriptstyle{u \geq 0, v>0, w > 0,} \\ [-5pt]
\scriptstyle{u + v +w = s} \\
\end{array}
}
\frac{k! p_0^u p_1^{k_1} \cdots p_{n-2}^{k_{n-2}} p_{n-1}^{v} p_{n}^{w}}
{u!k_1!\cdots k_{n-2}!v!w!} \\
& = \frac{k! p_1^{k_1} \cdots p_{n-2}^{k_{n-2}}}
{k_1!\cdots k_{n-2}!} \sum_{
\begin{array}{c}
\scriptstyle{u \geq 0, v>0, w > 0,} \\ [-5pt]
\scriptstyle{u + v +w = s} \\
\end{array}
}
\frac{p_0^u p_{n-1}^{v} p_{n}^{w}}
{u!v!w!} \\
& = \frac{k!
p_1^{k_1} \cdots p_{n-2}^{k_{n-2}}
\left(1 - (p_1 + \cdots + p_{n-2})\right)^{s}}
{k_1!\cdots k_{n-2}!s!}
\sum_{
\begin{array}{c}
\scriptstyle{u \geq 0, v>0, w > 0,} \\ [-5pt]
\scriptstyle{u + v +w = s} \\
\end{array}
}
\frac{s!}{u!v!w!} q_0^u q_{n-1}^v q_n^w \\
& = \frac{k! p_1^{k_1} \cdots p_{n-2}^{k_{n-2}}
\left(1 - (p_1 + \cdots + p_{n-2})\right)^{s}}
{k_1!\cdots k_{n-2}!s!}
\left(1 - \left(q_0+q_{n-1}\right)^{s}
- \left(q_0+q_{n}\right)^{s} 
+ q_0^{s}\right). \\
\end{align*}
Note that this relation is not true if at least one of the $k_\ell$ is zero.
Indeed, if $k_\ell=0$ for some $\ell = 1, \ldots,n-2$, we have
$$\Pr\{T_{n,n}(p) \leq k,N_1^{(k)}=k_1,\ldots,N_{n-2}^{(k)}=k_{n-2}\} = 0.$$
Summing over all the $k_1,\ldots,k_{n-2}$ such that 
$k_1 + \cdots + k_{n-2} \leq k$, we get
\begin{equation} \label{eqprem}
\Pr\{T_{n,n}(p) \leq k\} = 
\hspace{-0.5cm} \sum_{(k_1,\ldots,k_{n-2}) \in E_{n-2}} \hspace{-0.5cm}
\frac{k! p_1^{k_1} \cdots p_{n-2}^{k_{n-2}}
\left(1 - (p_1 + \cdots + p_{n-2})\right)^{s}}
{k_1!\cdots k_{n-2}!s!}
\left(1 - \left(q_0+q_{n-1}\right)^{s}
- \left(q_0+q_{n}\right)^{s} 
+ q_0^{s}\right),
\end{equation}
where the set $E_{n-2}$ is defined by
$$E_{n-2} = \{(k_1,\ldots,k_{n-2}) \in \left(\nbN^*\right)^{n-2} \mid
k_1 + \cdots + k_{n-2} \leq k\}$$
and $\nbN^*$ is the set of positive integers.

Note that for $n =2$, since $p_0+p_1+p_2=1$, we have 
$\Pr\{T_{2,2}(p) \leq k\} = \left(1 - \left(p_0+p_{1}\right)^{k}
- \left(p_0+p_{2}\right)^{k} 
+ p_0^{k}\right).$

Recall that $p_0 =1 - (p_1 + \cdots +p_n)$.
By definition of $p'_{n-1}$ and $p'_n$, we have,
for every $\lambda \in (0,1)$, $p'_{n-1} + p'_n = p_{n-1} + p_n$.
It follows that, by definition of $p'$,
$$p'_0 = 1 - (p_1 + \cdots + p_{n-2} + p'_{n-1} + p'_n)
= 1 - (p_1 + \cdots + p_{n-2} + p_{n-1} + p_n) = p_0.$$

Suppose that we have $p_{n-1} < p_n$.
This implies, by definition of $p'_{n-1}$ and $p'_n$, that
$p_{n-1} < p'_{n-1},p'_n < p_n$, that is
$q_{n-1} < q'_{n-1},q'_n < q_n$,
where
$$q'_{n-1} = \frac{p'_{n-1}}{p'_0+p'_{n-1}+ p'_n} 
= \frac{p'_{n-1}}{p_0+p_{n-1}+ p_n}
\mbox{ and }
q'_n = \frac{p'_{n}}{p'_0+p'_{n-1}+p'_n}
= \frac{p'_{n}}{p_0+p_{n-1}+p_n}.$$
In the same way, we have
$$q'_0 = \frac{p'_{0}}{p'_0+p'_{n-1}+ p'_n} = 
\frac{p_{0}}{p_0+p_{n-1}+ p_n} = q_0.$$
We thus have $q_0 + q_{n-1} < q'_0 + q'_{n-1},q'_0+q'_n < q_0+q_n$.
The function $f$ defined by $f(x) = x^s$ is convex on interval $[0,1]$ so,
from Lemma \ref{lemmconvex}, since 
$2q_0 + q_{n-1} + q_n = 2q'_0 + q'_{n-1} + q'_n$,
we have
\begin{equation} \label{inegq}
\left(q'_0+q'_{n-1}\right)^s +
\left(q'_0+q'_{n}\right)^{s}
\leq \left(q_0+q_{n-1}\right)^s + 
\left(q_0+q_n\right)^s.
\end{equation}
Similarly, if $p_n < p_{n-1}$, we have, by definition,
$p_{n} < p'_{n},p'_{n-1} < p_{n-1}$, that is
$q_{n} < q'_{n},q'_{n-1} < q_{n-1}$
and thus we also have Relation (\ref{inegq}) in this case.
Using Relation (\ref{inegq}) in Relation (\ref{eqprem}), we get, 
since $q'_0=q_0$,
\begin{align*}
\Pr\{T_{n,n}(p) \leq k\} & \leq 
\hspace{-0.5cm} \sum_{(k_1,\ldots,k_{n-2}) \in E_{n-2}} \hspace{-0.5cm}
\frac{k! p_1^{k_1} \cdots p_{n-2}^{k_{n-2}}
\left(1 - (p_1 + \cdots + p_{n-2})\right)^{s}}
{k_1!\cdots k_{n-2}!s!}
\left(1 - \left(q'_0+q'_{n-1}\right)^{s}
- \left(q'_0+q'_{n}\right)^{s} 
+ {q'_0}^{s}\right) \\
& = \Pr\{T_{n,n}(p') \leq k\},
\end{align*}
which completes the proof.
\end{proof}

The function $\Pr\{T_{n,n}(p) \leq k\}$, as a function of $p$, being 
symmetric, this theorem can easily 
be extended to the case where the two entries $p_{n-1}$ and $p_n$ of $p$,
which are different from the entries $p'_{n-1}$ and $p'_n$ of $p'$, 
are any $p_i, p_j \in \{p_1,\ldots,p_n\}$, with $i \neq j$.
 
In fact, we have shown in this theorem that for fixed $n$ and $k$, 
the function of $p$, $\Pr\{T_{n,n}(p) \leq k\}$, is a Schur-convex function,
that is, a function that preserves the order of majorization.
See \cite{Marshall81} for more details on this subject.

\begin{theorem} \label{notetheo1}
For every $n \geq 1$ and $p=(p_1,\ldots,p_n) \in (0,1)^n$ with
$p_1 + \cdots + p_n \leq 1$, we have, for every $k \geq 0$,
$$\Pr\{T_{n,n}(p) > k\} \geq \Pr\{T_{n,n}(v) > k\} \geq \Pr\{T_{n,n}(u) > k\},$$
where $v=(v_1,\ldots,v_n)$ with $v_i = (1-p_0)/n$ and
$p_0 = 1 - (p_1 + \cdots + p_n)$ and where $u=(1/n,\ldots,1/n)$.
\end{theorem}
 
\begin{proof}
To prove the first inequality, we apply successively and at most $n-1$ 
times Theorem \ref{autreloi}
as follows. We first choose two different entries 
of $p$, say $p_i$ and $p_j$ such that $p_i < (1-p_0)/n < p_j$ and next to
define $p'_i$ and $p'_j$ by
$$p'_i = \frac{1-p_0}{n} \mbox{ and } p'_j = p_i + p_j - \frac{1-p_0}{n}.$$
With respect to Theorem \ref{autreloi}, this leads to write
$p'_i = \lambda p_i + (1-\lambda) p_j$ and
$p'_j = (1-\lambda) p_i + \lambda p_j$, with  
$$\lambda = \frac{\DS{p_j - \frac{1-p_0}{n}}}{p_j - p_i}.$$
From Theorem \ref{autreloi}, the vector $p'$ that we obtain by taking the 
other entries equal to those of 
$p$, \emph{i.e.}, by taking $p'_\ell = p_\ell$, for $\ell = i,j$, is such that
$$\Pr\{T_{n,n}(p) > k\} \geq \Pr\{T_{n,n}(p') > k\}.$$
Note that at this point vector $p'$ has at least one entry equal to
$(1-p_0)/n)$, so repeating at most $n-1$ this procedure, we get vector $v$.

To prove the second inequality, we use Relation (\ref{multinome}). 
Introducing, for every $n \geq 1$, the set $F_n$ defined by 
$$F_{n}(\ell) = \{(k_1,\ldots,k_{n}) \in \left(\nbN^*\right)^{n} \mid
k_1 + \cdots + k_{n} = \ell\}.$$
For $k < n$, both terms are zero, so we suppose that $k \geq n$.
We have
\begin{eqnarray*}
\Pr\{T_{n,n}(v) \leq k\} 
& = & \Pr\{N_1^{(k)} > 0,\ldots,N_n^{(k)} > 0\} \\
& = & \sum_{k_0=0}^{k-n} 
      \Pr\{N_0^{(k)} = k_0, N_1^{(k)} > 0,\ldots,N_n^{(k)} > 0\} \\
& = & \sum_{k_0=0}^{k-n} \sum_{(k_1,\ldots,k_n)\in F_n(k-k_0)} 
      \frac{k!}{k_0!k_1!\cdots k_n!} 
      p_0^{k_0}\left(\frac{1-p_0}{n}\right)^{k-k_0} \\
& = & \sum_{k_0=0}^{k-n} {k \choose k_0} p_0^{k_0} (1 - p_0)^{k-k_0}
\frac{1}{n^{k-k_0}} \sum_{(k_1,\ldots,k_{n}) \in F_n(k-k_0)}
\frac{(k-k_0)!}{k_1!\cdots k_n!}. 
\end{eqnarray*}
Setting $p_0=0$, we get
$$\Pr\{T_{n,n}(u) \leq k\} = 
\frac{1}{n^{k}} \sum_{(k_1,\ldots,k_{n}) \in F_n(k)}
\frac{k!}{k_1!\cdots k_n!}.$$
This leads to
\begin{eqnarray*}
\Pr\{T_{n,n}(v) \leq k\} 
& = & \sum_{k_0=0}^{k-n} {k \choose k_0} p_0^{k_0} (1 - p_0)^{k-k_0} 
\Pr\{T_{n,n}(u) \leq k-k_0\} \\
& \leq & \Pr\{T_{n,n}(u) \leq k\}\sum_{k_0=0}^{k-n} {k \choose k_0}
p_0^{k_0} (1 - p_0)^{k-k_0} \\
& \leq & \Pr\{T_{n,n}(u) \leq k\},
\end{eqnarray*}
which completes the proof.
\end{proof}

To illustrate the steps used in the proof of this theorem, 
we take the following example. Suppose that $n=5$ and
$p = (1/16,1/6,1/4,1/8,7/24)$. This implies that $p_0 = 5/48$ and
$(1-p_0)/n = 43/240$.
In a first step, taking $i=4$ and $j=5$, we get
$$p^{(1)}=(1/16,1/6,1/4,43/240,19/80).$$
In a second, taking $i=2$ and $j=5$, we get
$$p^{(2)}=(1/16,43/240,1/4,43/240,9/40).$$
In a third step, taking $i=1$ and $j=3$, we get
$$p^{(3)}=(43/240,43/240,2/15,43/240,9/40).$$
For the fourth and last step, taking $i=5$ and $j=3$, we get
$$p^{(4)}=(43/240,43/240,43/240,43/240,43/240) = 
\frac{43}{48}(1/5,1/5,1/5,1/5,1/5).$$

\section{A new conjecture}
In this section, we propose a new conjecture stating that the complementary
distribution function of $T_{c,n}$ is minimal when the distribution $p$
is equal to the uniform distribution $u$. 

\begin{conj} 
For every $n \geq 1$, $c = 1, \ldots,n$ and $p=(p_1,\ldots,p_n) \in (0,1)^n$ 
with
$p_1 + \cdots + p_n \leq 1$, we have, for every $k \geq 0$,
$$\Pr\{T_{c,n}(p) > k\} \geq \Pr\{T_{c,n}(v) > k\} \geq \Pr\{T_{c,n}(u) > k\},$$
where $v=(v_1,\ldots,v_n)$ with $v_i = (1-p_0)/n$ and
$p_0 = 1 - (p_1 + \cdots + p_n)$ and where $u=(1/n,\ldots,1/n)$.
\end{conj}

This new conjecture is motivated by the 
following facts:
\begin{itemize}
\item the result is true for the expectations, see 
Theorem \ref{theocn}.
\item the result is true for $c=n$, see Theorem \ref{notetheo1}.
\item the result is trivially true for $c=1$ since
$$\Pr\{T_{1,n}(p) > k\} = \Pr\{T_{1,n}(v) > k\} = p_0^k \geq 1_{\{k=0\}}
= \Pr\{T_{1,n}(u) > k\}.$$ 
\item the result is true for $c=2$, see Theorem \ref{minic=2} below. 
\end{itemize}

\begin{theorem} \label{minic=2}
For every $n \geq 2$ and $p=(p_1,\ldots,p_n) \in (0,1)^n$
with
$p_1 + \cdots + p_n \leq 1$, we have, for every $k \geq 0$,
$$\Pr\{T_{2,n}(p) > k\} \geq \Pr\{T_{2,n}(v) > k\} \geq \Pr\{T_{2,n}(u) > k\},$$
where $v=(v_1,\ldots,v_n)$ with $v_i = (1-p_0)/n$ and
$p_0 = 1 - (p_1 + \cdots + p_n)$ and where $u=(1/n,\ldots,1/n)$.
\end{theorem}

\begin{proof}
From Relation (\ref{theo1}), we have
$$\Pr\{T_{2,n}(p) > k\} = -(n-1)p_0^k + \sum_{\ell = 1}^n (p_0 + p_\ell)^k$$
and
$$\Pr\{T_{2,n}(v) > k\} = -(n-1)p_0^k + n\left(p_0+\frac{1-p_0}{n}\right)^k.$$
For every constant $a \geq 0$, the function $f(x) = (a+x)^k$ is a convex on
$[0,\infty[$, so we have, taking $a = p_0$, by the Jensen 
inequality
$$\left(p_0+\frac{1-p_0}{n}\right)^k =
\left(\frac{1}{n}\sum_{\ell=1}^n (p_0 + p_\ell)\right)^k
\leq \frac{1}{n}\sum_{\ell=1}^n (p_0 + p_\ell)^k.$$
This implies that $\Pr\{T_{2,n}(p) > k\} \geq \Pr\{T_{2,n}(v) > k\}$.

To prove the second inequality, we define the function $F_{n,k}$ on interval 
$[0,1]$ by
$$F_{n,k}(x) = -(n-1)x^k + n\left(x+\frac{1-x}{n}\right)^k.$$
We then have $F_{n,k}(p_0) = \Pr\{T_{2,n}(v) > k\}$ and
$F_{n,k}(0) = \Pr\{T_{2,n}(u) > k\}$.
The derivative of function $F_{n,k}$ is
$$F'_{n,k}(x) = k(n-1)\left[\left(x+\frac{1-x}{n}\right)^{k-1} - 
x^{k-1}\right] \geq 0.$$
Function $F_{n,k}$ is thus an increasing function, which means that
$\Pr\{T_{2,n}(v) > k\} \geq \Pr\{T_{2,n}(u) > k\}$.
\end{proof}

\section{Application to the detection of distributed deny of service attacks}
\label{sec:application}

A Deny of Service (DoS) attack tries to  progressively take  down an Internet 
resource by flooding this resource with more requests than it is capable to 
handle. A Distributed Deny of Service (DDoS) attack is a DoS attack triggered 
by thousands of machines that have been infected by a malicious software, 
with as immediate consequence the total shut down of targeted web resources 
(\emph{e.g.}, e-commerce websites).
A solution to detect and  to mitigate DDoS attacks it to monitor network 
traffic at routers and to look for highly frequent signatures that might 
suggest  ongoing attacks. 
A recent strategy followed by the attackers is to hide their massive flow of 
requests over a multitude of routes, so that locally, these flows do not 
appear as frequent, while globally they represent a significant portion of 
the network traffic. The term ``iceberg'' has been recently introduced to describe such an attack as only a very small part of the iceberg can be observed from each single router. The approach adopted to defend against such new attacks is to rely on multiple routers that locally  monitor their network traffic, and  upon detection of potential icebergs,  inform a monitoring server that aggregates all the monitored information  to accurately detect icebergs.  Now to prevent the server from being overloaded by all the monitored information, routers continuously keep track of the $c$ (among $n$) most recent high flows (modelled as  items) prior to sending them to the server, and throw away all the items that appear with a small probability $p_i$, and such that the sum of these small probabilities is modelled by probability $p_0$. Parameter $c$ is dimensioned so that the frequency at which all the routers send their $c$ last frequent items is low enough to enable the server  to aggregate  all of them and to trigger a DDoS alarm when needed.  This  amounts to compute the  time needed to collect $c$ distinct items  among  $n$ frequent ones. Moreover, Theorem 5 shows that the expectation of this  time is minimal when the distribution of the frequent items is uniform.

\end{document}